\documentclass{amsart}
\usepackage{bm}
\usepackage{amssymb}
\usepackage{amsthm}
\usepackage{mathtools}
\usepackage{float}
\usepackage{tikz}
\usepackage{tikz-cd}
\usepackage{bbm} 
\usepackage{hyperref}
\usepackage{verbatim}
\usepackage{hyperref}
\hypersetup{
	colorlinks=true,
	linkcolor=red,
	filecolor=magenta,      
	urlcolor=cyan,
	pdftitle={Overleaf Example},
	pdfpagemode=FullScreen,
}
\usepackage[capitalise]{cleveref}

\makeatletter
\newtheorem*{rep@theorem}{\rep@title}
\newcommand{\newreptheorem}[2]{%
	\newenvironment{rep#1}[1]{%
		\def\rep@title{#2 \ref{##1}}%
		\begin{rep@theorem}}%
		{\end{rep@theorem}}}
\makeatother

\newtheorem{theorem}{Theorem}[section]
\newreptheorem{theorem}{Theorem}
\newtheorem*{theorem*}{Theorem}
\newtheorem{corollary}[theorem]{Corollary}
\newreptheorem{corollary}{Corollary}
\newtheorem{lemma}[theorem]{Lemma}

\newtheorem{proposition}[theorem]{Proposition}
\newreptheorem{proposition}{Proposition}

\newtheorem*{claimn*}{Claim}

\theoremstyle{definition}
\newtheorem{definition}[theorem]{Definition} 
\newtheorem{remark}[theorem]{Remark}

\newtheorem{notation}[theorem]{Notation}

\newcommand{\cantor}{2^{\omega}}
\newcommand{\baire}{\omega^{\omega}}

\newcommand{\bSigma}[2]{\bm{\Sigma}^{#1}_{#2}}
\newcommand{\bPi}[2]{\bm{\Pi}^{#1}_{#2}}

\newcommand{\lSigma}[2]{\Sigma^{#1}_{#2}}
\newcommand{\lPi}[2]{\Pi^{#1}_{#2}}
\newcommand{\lDelta}[2]{\Delta^{#1}_{#2}}

\newcommand{\clo}{\mathrm{cl}_{\omega}}
\newcommand{\open}[1]{[#1]^{\prec}}
\newcommand{\concat}[2]{#1^{\smallfrown}#2}
\newcommand{\stem}[1]{\mathrm{stem}(#1)}

\newcommand{\T}{\mathrm{T}}

\newcommand{\ideal}[1]{\mathcal{#1}}
\newcommand{\ock}{\omega^{\mathrm{CK}}_1}
\newcommand{\ko}{\mathcal{O}}

\newcommand{\zfc}{\ensuremath{\mathsf{ZFC}}}

\newcommand{\finf}{\forall^\infty}

\newcommand{\hyp}{\mathsf{HYP}}
\newcommand{\lhyp}{\leq_{\hyp}}
\newcommand{\wl}[1]{#1\text{-}\mathsf{WL}}
\newcommand{\li}[1]{#1\text{-}\mathsf{L}}
\newcommand{\SME}[1]{#1\text{-}\mathsf{SME}}
\newcommand{\SNE}[1]{#1\text{-}\mathsf{SNE}}
\newcommand{\domm}[1]{#1\text{-}\mathsf{DOM}}

\newcommand{\seq}[2]{\langle #1 : #2 \rangle}

\newcommand{\forces}{\Vdash}
\newcommand{\use}[1]{\mathrm{use}(#1)}

\title{Listing the hyperarithmetical functions}
\author{Joseph S. Miller} 
\author{Gian Marco Osso}
\author{Isabella Scott}

\begin{document}
	
	\begin{abstract}
		Given a countable Turing ideal $\ideal{I} \subseteq \baire$, we say that $x$ is a \emph{list} (resp.\ \emph{weak list}) of $\ideal{I}$ if $\ideal{I}=\{x^{[n]} : n \in \omega\}$ (resp.\ if $\ideal{I} \subseteq \{x^{[n]} :n \in \omega\}$). We show that, for several natural ideals $\ideal{I}$, $x$ computes a list of $\ideal{I}$ if and only if it computes a function dominating all the functions in $\ideal{I}$. On the other hand, we provide reals which are $\hyp$-strongly null engulfing (and hence $\hyp$-dominating, by results of \cite{GKT}) but which cannot compute a weak list for $\hyp$, solving a problem left open in \cite{GO}.  This result can be generalized to any countable ideal which is downward closed under $\lhyp$. We also give a characterization of reals which compute a list of $\hyp$: $x$ computes a list of $\hyp$ if and only if $x$ is $\hyp$-dominating and $\ko$ is $\lSigma{0}{2}(x)$.
	\end{abstract}
	\maketitle
	\section{Introduction}
	
	This paper solves a problem which comes up naturally in \cite{GO} in the context of computable analogues of cardinal characteristics. The latter paper builds on work of \cite{GKT} (see also \cite{nies}), showing how one can associate to any countable Turing ideal $\ideal{I}$ the so-called \emph{$\ideal{I}$-non-lowness classes}, corresponding to computational analogues of cardinal characteristics (the classes of interest in \cite{GKT} are specifically those corresponding to cardinals in Cicho\'{n}'s diagram). It is shown in \cite{GKT} that, whenever we have a $\zfc$ provable inequality $A \leq B$ between cardinal characteristics in Cicho\'{n}'s diagram, then, for every $\ideal{I}$, the $\ideal{I}$-non-lowness class of $A$ is contained in the $\ideal{I}$-non lowness class of $B$. 
	
	In \cite{GKT}, the authors single out the classes $\domm{\ideal{I}}$ (\emph{$\ideal{I}$-dominating}), $\SME{\ideal{I}}$ (\emph{$\ideal{I}$-strongly meagre engulfing}) and $\SNE{\ideal{I}}$ (\emph{$\ideal{I}$-strongly null engulfing}), defined below, for their interesting behavior. According to the $\zfc$-provable inequalities between the associated cardinals, the inclusions $\SNE{\ideal{I}} \subseteq \SME{\ideal{I}} \subseteq \domm{\ideal{I}}$ hold for every countable Turing ideal. Set theory would also predict that $\SNE{\ideal{I}} \subsetneq \SME{\ideal{I}} \subsetneq \domm{\ideal{I}}$, but, for example, we have $\SNE{\mathsf{REC}}=\SME{\mathsf{REC}}=\domm{\mathsf{REC}}$. Analyzing the proofs of these facts, one observes the following: (i) a real $x$ which dominates all the computable functions is actually capable of \emph{listing} all the computable functions (see \cite{jockusch}) and (ii) a real $x$ that can list an ideal $\ideal{I}$ is in $\SNE{\ideal{I}}$ (see \cite[Theorem 6]{rupprecht}, or \cite[Theorem 2]{nies} for a more direct proof that listing reals are strongly meagre engulfing. The latter proof can be easily translated from the context of category to that of measure). This readily explains the unexpected inclusion $\domm{\mathsf{REC}} \subseteq \SNE{\mathsf{REC}}$ and hence the collapse of the entire chain. 
	
	In \cite{GO}, the authors were able to prove that, if $\ideal{I}$ is downward closed under $\lhyp$, then $\SNE{\ideal{I}} \subsetneq \SME{\ideal{I}} \subsetneq \domm{\ideal{I}}$. It immediately follows that, for such $\ideal{I}$, not all reals $x$ which dominate all the functions in $\ideal{I}$ can also list $\ideal{I}$. This paper addresses the natural question: what conditions on $\ideal{I}$ are sufficient (and perhaps necessary) to allow for some $x \in \SNE{\ideal{I}}$ that cannot list $\ideal{I}$?
	
	For any countable Turing ideal $\ideal{I}$, we say that $x$ is a \emph{weak list} of $\ideal{I}$ if, for every $y \in \ideal{I}$, there is some $n \in \omega$ with $x^{[n]}=y$, and we say that $z$ is a \emph{list} of $\ideal{I}$ if $\{z^{[n]} : n \in \omega\}=\ideal{I}$. Our first theorem is
	
	\begin{theorem*}[\cref{thm:list}]
		For any $\alpha < \ock$, let $\ideal{I}_{\alpha}=\{x \in \baire : \exists \beta < \alpha \,\, (x \leq_\T \emptyset^{(\beta)})\}$. Any $f \in \baire$ which dominates all the functions of $\ideal{I}_{\alpha}$ computes a list of $\ideal{I}_{\alpha}$.
	\end{theorem*}
	
	It follows from the discussion above that $\SNE{\ideal{I}_\alpha}=\SME{\ideal{I}_\alpha}=\domm{\ideal{I}_\alpha}$. 
	
	Moreover, we characterise the reals $x$ that compute a list of $\hyp$.
	
	\begin{theorem*}[\cref{thm:hyplist}]
		A real $x$ computes a list of $\hyp$ if and only if $x$ computes a function which dominates $\hyp$ and $\ko$ is $\lSigma{0}{2}(x)$.
	\end{theorem*}
	
	 By contrast we observe that, by definition, the set of weak lists of $\hyp$ is $\lSigma{1}{1}$, hence by Gandy's basis theorem there is some weak list $y$ with $y <_{\hyp} \ko$. This immediately shows that there are some weak lists of $\hyp$ which cannot compute any list of $\hyp$. 
	 
	 Finally, we extend the analysis of \cite{GO}, by showing that, for all countable Turing ideals $\ideal{I}$ which are closed under $\lhyp$, there are $\ideal{I}$-strongly null engulfing reals which cannot compute any weak list of $\ideal{I}$ (not even with the help of an oracle in $\ideal{I}$). 
	 
	 \begin{theorem*}[\cref{thm:sep}]
	 	Let $\ideal{I}$ be a countable Turing ideal which is $\lhyp$-downward closed. We have $\wl{\ideal{I}} \subsetneq \SNE{\ideal{I}}$.
	 \end{theorem*}
	
	\section{Preliminaries}
	
	Our notation is mostly standard in computability theory. We use $\lambda$ to denote the Lebesgue measure on $\cantor$. If $\sigma$ is a finite string and $\tau$ is a finite or infinite string, we denote the concatenation of $\sigma$ and $\tau$ as $\concat{\sigma}{\tau}$. We write $\sigma \preceq \tau$ to denote that $\sigma$ is a prefix of $\tau$. If $\sigma \in 2^{<\omega}$, we write $[\sigma]$ to denote the basic clopen set in Cantor space given by $\{f \in \cantor : \sigma \prec f\}$. If $f,g \in \baire$, we write $f \geq g$ (read: $f$ \emph{majorizes} $g$) as an abbreviation for $\forall i \,\, (f(i) \geq g(i))$, and we write $f \geq^* g$ (read: $f$ \emph{dominates} $g$) as an abbreviation for $\exists n\,\, \forall i \geq n \,\, (f(i) \geq g(i))$. We write $\langle \rangle$ to denote a standard pairing function, i.e.\ a computable bijection $\omega^2 \rightarrow \omega$. This pairing function allows us to divide each real into countably many columns as follows: given $x \in \baire$ (or $x \in \cantor$), the $n$-th column of $x$ is $x^{[n]}$ defined as $x^{[n]}(k)=x(\langle n, k \rangle)$. We use columns to work with infinite joins: given a sequence $\seq{x_n}{n \in \omega}$ of reals, $x=\bigoplus_{n \in \omega} x_n$ is the unique real such that $x^{[n]}=x_n$. We write $\varphi_n$ to denote the $n$-th partial function of type $\omega \rightarrow \omega$. The Turing functional of type $\baire \rightarrow \baire$ induced by $\varphi_n$ is denoted $\Phi_n$. We denote the \emph{use} of the computation of the $e$-th computable function, with oracle $A$, on input $n$ as $\use{e, n, A}$. We adopt the convention that if $\varphi^A_e(n) \downarrow$, then $\use{e, n, A}$ is the maximum between $e$, $n$, the number of steps of computation necessary to see convergence, and the maximum index of $A$ queried during the computation. If $\varphi^A_e(n) \uparrow$, then $\use{e, n, A}=0$. A Turing ideal is a set $\ideal{I} \subseteq \baire$ which is closed under Turing join ($\oplus$) and is downward closed under Turing reduction (i.e.\ $(x \in \ideal{I} \land y \leq_\T x) \rightarrow y \in \ideal{I}$). We denote Kleene's set of notations for computable ordinals as $\ko$, and its relativization to oracle $x$ as $\ko^x$. If $a \in \ko^x$ and $x \in \cantor$ (or $x \in \baire$) we denote by $x^{(a)}$ the result of iterating the Turing jump along the computable well-order indexed by $a$, starting from $x$. Similarly, if $\alpha < \omega^x_1$ and $a \in \ko^x$ is a notation for $\alpha$, we may abuse notation and write $x^{(\alpha)}$ to denote any set or function Turing equivalent to $x^{(a)}$. We assume the reader is familiar with the basic notions of computability and hyperarithmetic theory. Standard references for these subjects are \cite{rogers} and \cite{sacks}.
	
	\subsection{Basic definitions}
	
	We now give definitions for the notions of Schnorr randomness, domination and listing relative to Turing ideals.
	
	\begin{definition}
		A \emph{representation} of a set $X$ is a partial map $\pi: \omega^\omega \to X$.  A \emph{$\pi$-name} for $x \in X$ (or often merely \emph{name}, when $\pi$ is clear from context) is $a \in \omega^\omega$ such that $\pi(a) = x$.
	\end{definition}
	
	We now give representations for the first levels of the Borel hierarchy on $\cantor$. For the rest of the paper, we think we have fixed a canonical bijection $i \colon \omega \rightarrow 2^{<\omega}$.
	
	\begin{definition}
		A name for $\bSigma{0}{1}$ set $U \subseteq 2^\omega$ is a function $a \in \baire$ with $U = \bigcup_{n \in \omega} [f(i(n))]$, and
		a name for a $\bPi{0}{2}$ set $X$ is a sequence of names for open sets $U_n$ such that $X = \bigcap_{n \in \omega} U_n$.
	\end{definition}
	
	Note that these representations both have $\baire$ as domain. For any $f \in \baire$, we denote by $\open{f}$ the open set named by $f$.
	
	We code real numbers as fast Cauchy sequences. For the following definition, let $p$ be a standard bijection from $\omega$ into $\mathbb{Q}$.
	
	\begin{definition}
		We define the representation $\pi_{\mathbb{R}} \colon \baire \rightarrow \mathbb{R}$ as follows: a name for a real number $r$ is a sequence $f$ such that $\lim_{i}p(f(i))=r$ and, for all $i < j$, $|p(f(i))-p(f(j))| < 2^{-i}$.
	\end{definition}
	
	This representation and its basic properties are well known, and we will deal with it quite informally.
	
	\begin{definition}
		A name for a \emph{Schnorr null} set $N \subseteq \cantor$ is a sequence $\seq{(f_n, \lambda_n)}{n \in \omega}$ such that $\seq{f_n}{n \in \omega}$ is a name for $N$ as a $\bPi{0}{2}$ set, $\lambda(\open{f_n}) = \pi_{\mathbb{R}}(\lambda_n)$ for every $n$, and $\lim_{n} \pi_{\mathbb{R}}(\lambda_n)=0$. 
	\end{definition}
	
	If $\ideal{I}$ is a Turing ideal, we say that an open (resp.\ $\bPi{0}{2}$, Schnorr null) set $A$ is coded in $\ideal{I}$ if $A$ has a name as an open (resp.\ as a $\bPi{0}{2}$, Schnorr null) set in $\ideal{I}$.
	
	\begin{remark}\label{rem:names}
		Given any name $x=\seq{(f_n, \lambda_n)}{n \in \omega}$ for a Schnorr null set $\bigcap_{n \in \omega} \open{f_n}$, there is a Turing equivalent name (a subsequence of $x$) $\seq{f_{n_k}, \lambda_{n_k}}{k \in \omega}$ with $\lambda_{n_k} \leq 2^{-k-1}$.
	\end{remark}
	
	In line with \cite{GO}, given a Turing ideal $\ideal{I}$, we will often prove results from the point of view of $\ideal{I}$-computability.
	
	\begin{definition}
		Let $\ideal{I}$ be a Turing ideal, $x, y \in \baire$. We write $x \leq^{\ideal{I}} y$ (read: $y$ $\ideal{I}$-computes $x$) if there is some $z \in \ideal{I}$ such that $x \leq_\T z \oplus y$.
	\end{definition}
	
	We are now ready to give definitions for the classes that we will work on.
	
	\begin{definition}
		A real $x$ is \emph{$\ideal{I}$-strongly null engulfing} if it can $\ideal{I}$-compute a name for a Schnorr null set $N$ which contains all the $\ideal{I}$-coded Schnorr null sets. We denote the class of $\ideal{I}$-strongly null engulfing reals as $\SNE{\ideal{I}}$.
	\end{definition}
	
	\begin{definition}
		A real $x$ is \emph{$\ideal{I}$-dominating} if it can $\ideal{I}$-compute a function $f$ in $\baire$ such that $f \geq^* g$ for every $g \in \ideal{I}$. We denote the class of $\ideal{I}$-dominating reals as $\domm{\ideal{I}}$.
	\end{definition}
	
	\begin{definition}
		A real $x$ is a \emph{weak list} of a countable ideal $\ideal{I}$ if for every $g \in \ideal{I}$ there is some $n$ with $x^{[n]}=g$, and it is a \emph{list} of $\ideal{I}$ if $\{x^{[n]} : n \in \omega\}=\ideal{I}$.
		
		We define the classes of \emph{$\ideal{I}$-weak listing} and \emph{$\ideal{I}$-listing} reals (denoted $\wl{\ideal{I}}$ and $\li{\ideal{I}}$, respectively) as consisting of those $x$ which can $\ideal{I}$-compute a weak list (resp.\ a list) of $\ideal{I}$.
	\end{definition}
	
	By definition, for every countable Turing ideal $\ideal{I}$, we have $\li{\ideal{I}} \subseteq \wl{\ideal{I}}$. A simple diagonalization suffices to show that $\wl{\ideal{I}} \subseteq \domm{\ideal{I}}$. Indeed, the argument in \cite[Theorem 2]{nies}, translated in terms of measure, yields $\wl{\ideal{I}} \subseteq \SNE{\ideal{I}}$. Results of \cite{GKT} show that we always have $\SNE{\ideal{I}} \subseteq \SME{\ideal{I}} \subseteq \domm{\ideal{I}}$. As we mentioned above, the proofs in the literature showing that $\domm{\mathsf{REC}} \subseteq \SNE{\mathsf{REC}}$ actually rely on the fact (due to Jockusch, see \cite{jockusch}) that $\domm{\mathsf{REC}} \subseteq \li{\mathsf{REC}}$. Summarizing the situation for the ideal of computable functions, we have $\li{\mathsf{REC}}=\wl{\mathsf{REC}}=\SNE{\mathsf{REC}}=\domm{\mathsf{REC}}$.
	
	\section{When dominating and listing coincide}
	
	We show that the inclusion $\domm{\ideal{I}} \subseteq \li{\ideal{I}}$ holds for all the ideals which can be obtained as the $\leq_\T$-downward closure of chains given by iterating the Turing jump (starting from $\emptyset$) along some initial segment of the computable ordinals. This is an extension of the classic result of Jockusch that $\domm{\mathsf{REC}}=\li{\mathsf{REC}}$.
	
	\begin{proposition}[Jockusch]\label{prop:complist}
		$\domm{\mathsf{REC}}=\li{\mathsf{REC}}$.
	\end{proposition}
	\begin{proof}
		We only sketch the proof of $\domm{\mathsf{REC}} \subseteq \li{\mathsf{REC}}$. The converse inclusion is shown by a simple diagonalization.
		
		Suppose that $f \in \domm{\mathsf{REC}}$. Notice that if $\Phi_e(\emptyset)$ converges on all inputs (i.e.\ defines an element of $\baire$), then the function $u_e$ given by $n \mapsto \use{e,n, \emptyset}$ is total computable, so it is dominated by $f$. This allows us to use $f$ as a \emph{stopping rule}, as follows.
		
		We define a computable sequence of reals $\seq{x_{e,n}}{e,n \in \omega}$ by saying that $x_{e,n}(m)$ is computed by running $\varphi^{\emptyset}_e(m)$ for $n+f(m)$ steps. If, for some $m$ and $n$, the computation $\varphi^{\emptyset}_e(m)$ does not halt in $n+f(m)$ steps, we set $x_{e,n}(k)=0$ for every $k \geq m$. It is easy to see that $\seq{x_{e,n}}{e,n \in \omega}$ lists exactly the computable functions.
	\end{proof}
	By relativization, Proposition \ref{prop:complist} holds relative to any ideal $\ideal{I}$ with a top Turing degree.
	
	We now briefly review the notion of \emph{modulus of computation} for a real. This concept is key in understanding how one can extract computational power from functions which grow very fast. Moreover, we recall classical results about moduli of computation and hyperarithmetical sets, providing proofs for completeness.
	
	\begin{definition}
		Let $f \in \baire$. We say that $g \in \baire$ is a \emph{modulus of computation} for $f$ if, for all $h \geq g$ (equivalently, for all $h \geq^* g$) we have $f \leq_\T h$. We say that $g$ is a \emph{uniform} modulus of computation for $f$ if there is an index $e$ such that, for all $h \geq g$, $\Phi_e(h)=f$.
	\end{definition}
	
	We show that all the hyperarithmetical functions have uniform moduli of computation and, in particular, that each of the sets $\emptyset^{(\alpha)}$ with $\alpha < \ock$ has a \emph{self-modulus of computation}, i.e.\ a modulus of computation $g_{\alpha}$ with $g_{\alpha}=\emptyset^{(\alpha)}$.
	
	\begin{definition}
		We denote by $g_1$ the function $g_1(n)=\use{n, n, \emptyset}$.
	\end{definition}
	
	\begin{lemma}\label{lem:modofzeroprime}
		The function $g_1$ is computable in $\emptyset'$, and it is a uniform modulus of computation for $\emptyset'$.
	\end{lemma}
	\begin{proof}
		The fact that $g_1 \leq_\T \emptyset'$ is immediate by definition.
		
		To see that it is a uniform modulus of computation for $\emptyset'$, let $h \geq g_1$. For any $n$, to compute $\emptyset'(n)$ from $h$, we can simply try to run the computation $\varphi_n(n)$ for $h(n)$ steps. If that computation converges, then $n \in \emptyset'$. Otherwise we can safely say $n \notin \emptyset'$
	\end{proof}
	
	We now show by effective transfinite recursion that, for every $a \in \ko$, the set $\emptyset^{(a)}$ has a uniform modulus of computation. 
	
	\begin{proposition}[Folklore]\label{prop:higherjumps}
		There is a computable function $c \colon \omega \rightarrow \omega$ and a sequence $\seq{g_n}{n \in \omega}$ such that, for every $a \in \ko$, $c(a)$ is an index witnessing that $g_a$ is a uniform modulus for $\emptyset^{(a)}$. Moreover, if $a \in \ko$, then $g_a \equiv_{\T} \emptyset^{(a)}$.
	\end{proposition}
	\begin{proof}
		We make use of auxiliary functions $\ell_{2^n}(k)= \use{k, k,  \emptyset^{(n)}}$, where, if $n \notin \ko$, we stipulate that $\emptyset^{(n)}=\emptyset$. If $n \in \ko$, then $\emptyset^{(n)}$ has the usual meaning.
		
		Now we can define the function $c$ and the sequence $\seq{g_n}{n \in \omega}$. Notice that $c(0)$ can be taken to be any index for a function computing $\emptyset$ (ignoring its oracle) and we can set $g_0$ to be the constant $0$ function. The function $g_1$ and the index $c(1)$ are those of Lemma \ref{lem:modofzeroprime}. We define:		
		\begin{equation*}
			g_n=
			\begin{cases}
				g_m \oplus \ell_n & \text{if } n=2^m \\
				\bigoplus_{i \in \omega} g_{e(i)} & \text{if } n=3 \cdot 5^e \\
				g_0 & \text{otherwise.}
			\end{cases}
			\,\,\,\,\,\,\text{and}\,\,\,\,\,\,
			c(n)=
			\begin{cases}
				h(n) & \text{if } n=2^m \\
				j(n) & \text{if } n=3 \cdot 5^e \\
				e_0 & \text{otherwise.}
			\end{cases}
		\end{equation*}
		where $e_0$ is an index for an everywhere diverging function, and the computable functions $h$ and $j$ are defined as follows. If $n=2^m$ for some $m$, the computation of $\varphi^A_{h(n)}(k)$ interprets $A$ as $A_0 \oplus A_1$, runs $\varphi^{A_{0}}_{c(m)}$ as long as needed to run the computation of $\varphi^{\Phi_{c(m)}(A_0)}_{k}(k)$ with use given by $A_1(k)$. If the latter computation halts within these steps, the entire computation halts with value $1$. Otherwise, it halts with value $0$. It does not matter how $h$ is defined on values $n$ which are not powers of $2$, so we may assume that $h(n)=e_0$ in those cases. Similarly, we set $j(n)=e_0$ if $n$ is not of the form $3 \cdot 5^e$ for some $e$. If $n= 3 \cdot 5^e$ for some $e$, $\varphi^A_{j(n)}(k)$ interprets $A$ as $\bigoplus_{i \in \omega}A_i$, and, if $k=\langle i, m\rangle$, computes $\varphi^{A_i}_{c(i)}(m)$.
		
		Note that the function $c$ is defined using the Recursion Theorem.
		
		By induction on $<_{\ko}$, for every $a \in \ko$, we have that $c(a)$ is the index for a reduction $\emptyset^{(a)} \leq_\T g_a$. To see that we also have $g_a \leq_\T \emptyset^{(a)}$, we argue again by transfinite induction. Indeed, notice that if $a=2^n$ is a notation in $\ko$, then $\ell_a \leq_\T \emptyset^{(a)}$, and this is uniform in $a$. Hence, by definition of the $g_a$'s we have $g_a \equiv_{\T} \emptyset^{(a)}$ for every $a \in \ko$.
	\end{proof}
	
	\begin{remark}
		There is a crucial difference between the computation for $\emptyset'$ sketched in Lemma \ref{lem:modofzeroprime} and the computations for longer iterations of the Turing jump in Proposition \ref{prop:higherjumps}. In the former, we use $g_1$ as a bound on the total number of computation steps. In other words, $c(1)$ is the index for a truth table reduction.
		
		There is no analogous bound in the reductions of Proposition \ref{prop:higherjumps}. Consider for example the definition of $\varphi^A_{h(n)}(k)$, with $n=2^m$. It is possible that the computation $\Phi_{c(m)}(A_0)$ diverges, so that we will never be able to compute enough bits of the candidate oracle to carry out $A_1(k)$ steps of the computation $\varphi^{\Phi_{c(m)}(A_0)}_{k}(k)$. In this case the whole computation diverges.
		
		This is unavoidable, meaning that there is no\ there is no computable sequence $c'$ satisfying Proposition \ref{prop:higherjumps} and such that, for every $n$, $c'(n)$ is the index of a truth table reduction. This follows from Theorems \ref{thm:hyplist} and \ref{thm:sep}, together with the inclusion $\SNE{\hyp} \subseteq \domm{\hyp}$.
	\end{remark}
	
	\begin{corollary}
		Every hyperarithmetical function has a uniform modulus of computation.
	\end{corollary}
	\begin{proof}
		This follows immediately from the observation that the class of functions with a (uniform) modulus of computation is $\leq_\T$-downward closed.
	\end{proof}
	
	It is known that the hyperarithmetical functions are exactly those which have a modulus of computation, see \cite{hypenc}.
	
	We can finally prove the main result of this section.
	
	\begin{theorem}\label{thm:list}
		For any $\alpha < \ock$, let $\ideal{I}_{\alpha}=\{x \in \baire : \exists \beta < \alpha \,\, (x \leq_\T \emptyset^{(\beta)})\}$. Any $f \in \baire$ which dominates all the functions of $\ideal{I}_{\alpha}$ can compute a list of $\ideal{I}_{\alpha}$, so in particular $\domm{\ideal{I}_{\alpha}}=\li{\ideal{I}_\alpha}$.
	\end{theorem}
	\begin{proof}
		Fix $\alpha < \ock$, a notation $a \in \ko$ for $\alpha$, and $f$ dominating all functions in $\ideal{I}_{\alpha}$.

		Notice that if $\alpha=\beta+1$ is a successor, then $\ideal{I}_{\alpha}$ has a top element, namely $\emptyset^{(\beta)}$. In this case the result follows from (the relativization of) Proposition \ref{prop:complist}, so we may assume (although it is not necessary) that $\alpha$ is a limit ordinal.
		
		In view of (the relativization of) Proposition \ref{prop:complist}, it is sufficient to show that $f$ can compute a sequence $\seq{y_n}{n \in \omega} \subseteq \ideal{I_{\alpha}}$ which lists the sets $\seq{\emptyset^{(\beta)}}{\beta <\alpha}$.
		
		To show that, notice that the initial segment of $\ko$ given by $R_a=\{b \in \ko : b <_{\ko} a\}$ is computably enumerable, and for all $b \in R_a$, $f$ dominates the function $g_b$ (cf.\ Proposition \ref{prop:higherjumps}). Computing $\emptyset^{(b)}$ from $f$ is at first sight a non-uniform procedure, as it looks like we would need to know some sequence of numbers number $\seq{m_b}{b \in R_a}$ such that $f(i) \geq g_b(i)$ for every $i \geq m_b$ and every $b \in R_a$. We circumvent this problem with a construction like that of Proposition \ref{prop:complist}, with a ``finite injury'' flavor. 
		For $n \in \omega$, let $n + f$ denote the function $m \mapsto n +f(m)$. For every $b \in R_a$, since $f \geq^*g_b$, there is some $n$ such that $n+f \geq g_b$.
		
		We define a sequence $\seq{x_{b, n}}{b \in R_a, n \in \omega}$ as follows. We think that, for any fixed $b$, the bits of the reals the sequence $\seq{x_{b,n}}{n \in \omega}$ are being computed in parallel. We first attempt to make it so that, for every $n$ and $m$, $\varphi^{n+f}_{c(b)}(m)=x_{b,n}(m)$, but, if $n<k$, we ``trust'' the computation with oracle $k+f$ more than that with oracle $n+f$. We give the formal definition for a single $b$, as computations of reals indexed with different ordinal notations never interact.
		
		Fix $b \in R_a$. We compute in stages the sequence $\seq{x_{b,n}}{n \in \omega}$, as well as an auxiliary non-decreasing sequence $\seq{t_s}{s \in \omega}$ of numbers, which is needed to keep track of the $n$ such that we currently believe $n+f \geq g_b$. We start with $x_{b,n}[0]=\emptyset$ and $t_0=0$. At stage $s+1$, for every $b \in R_a$, every $t_s < n< s+1$ and every $m <s+1$ we check if $\varphi^{n+f}_{c(b)}(m)[s+1] \downarrow$.
		For a fixed $m$, if there are $t_s <n < k<s+1$ such that $\varphi^{k+f}_{c(b)}(m)[s+1]$ and $\varphi^{n+f}_{c(b)}(m)[s+1]$ converge to two different values, then we stop believing that $n+f$ majorizes $g_b$. If this happens, we say that $n$ and $k$ are in conflict (over $m$). In this case, we set all undefined bits of $x_{b,n}$ to $0$, and we ensure that $t_s \geq n$. We then let $n_m=\max\{n :\varphi^{n+f}_{c(b)}(m)[s+1] \downarrow\}$ and, for every $k \leq n_m$ such that $\varphi^{k+f}_{c(b)}(m)[s+1] \uparrow$,  if $k$ has never been in conflict with any $n>k$, we set $x_{b,k}(m)[s+1]=\varphi^{n_m+ f}_{c(b)}(m)[s+1]$. We let $t_{s+1}$ be the least $t \geq t_s$ such that, if $n$ is in conflict with some $k>n$ at stage $s+1$, then $t>n$.
		
		\begin{claimn*}
			The sequence $t_s$ is bounded, and, if $t= \lim_{s}t_s$ and $k>t$, then $\Phi^{k+f}_{c(b)}=\emptyset^{(b)}$.
		\end{claimn*}
		\begin{proof}[Proof of claim.]
			Let $m$ be such that $n+f$ majorizes $g_b$ for every $n \geq m$. By Proposition \ref{prop:higherjumps}, $\Phi^{n+f}_{c(b)}= \emptyset^{(b)}$ for all $n \geq m$, hence there will be no conflicts between indices above $m$. This immediately implies that $t_s \leq m$, so that $t_s$ has a limit, say $t$. Let $k > t$. By construction, $k$ is never in conflict with any $k'>k$, so in particular it is never in conflict with $m$. This implies that $\Phi^{k+f}_{c(b)}= \emptyset^{(b)}$.
		\end{proof}
		
		By construction the sequence $\seq{x_{b,n}}{n \leq t}$ consists of functions which are eventually $0$. Putting everything together we get that for every $b \in R_a$ and every $n \in \omega$, $x_{b,n}$ is either eventually $0$ or $\emptyset^{(b)}$, and for every $b$ there is $m$ such that $x_{b,m}=\emptyset^{(b)}$.
	\end{proof}
	
	\section{Domination and (weak) listing in $\hyp$}
	
	We show that $x$ computes a list of $\hyp$ if and only if $x$ computes a function which dominates $\hyp$ and $\ko$ is $\lSigma{0}{2}(x)$. The right to left implication is proved via an extension of the proof of Theorem \ref{thm:list} and it was previously known to Noam Greenberg and Dan Turetsky. This result is an ingredient in Turetsky's proof of a limitative result on arboreal forcing over $\hyp$ (see \cite[Appendix A]{GO}). As we mentioned in the introduction, this readily implies that $\li{\hyp} \subsetneq \wl{\hyp}$. Since $\wl{\hyp} \subseteq \domm{\hyp}$, this shows that the condition that $\ko$ is $\lSigma{0}{2}(x)$ cannot be omitted from the characterization of $x \in \li{\hyp}$. We also give a direct proof that $\wl{\hyp} \subsetneq \domm{\hyp}$, anticipating the technique we use to prove of Theorem \ref{thm:sep}.
	
	\begin{theorem}\label{thm:hyplist}
		A real $x$ computes a list of $\hyp$ if and only if $x$ computes a function which dominates $\hyp$ and $\ko$ is $\lSigma{0}{2}(x)$. In particular, $x \in \li{\hyp}$ if and only if $x \in \domm{\hyp}$ and there is $y \in \hyp$ such that $\ko$ is $\lSigma{0}{2}(x \oplus y)$.
	\end{theorem}
	\begin{proof}
		Let $f=\seq{f_n}{n \in \omega} \leq_\T x$ be a list of $\hyp$. An easy diagonalization shows that $f$, so also $x$, computes a function $g$ which dominates $\hyp$. Now note that $a \in \ko$ if and only if there is some $n$ such that $f_n= \emptyset^{(a)}$. By \cite[Theorem II.4.2]{sacks}, the expression $X = \emptyset^{(a)}$, where $a \in \ko$, is $\lPi{0}{2}$. This predicate can be seen as saying that, intuitively, the relationships between the columns of $X$ are those expected from a jump hierarchy, e.g.\ that, for every $b<_{\ko} 2^b <_{\ko} a$, $X^{[b]}=\Phi_k(X^{[2^b]})$, where $k$ is a fixed index which witnesses an $m$-reduction from $Y$ to $Y'$, and a similar statement for notations $3 \cdot 5^e \in \ko$ which denote limit ordinals. Expanding this out, we get that $a \in \ko$ if and only if 
		\[
		\chi(a) = \exists n \, \forall b,b', m \, \exists s \,\, ((b <_{\ko} b' <_{\ko} a) \rightarrow \varphi^{f^{[b']}_n}_{d(b,b')}(m)[s] \downarrow =f^{[b]}_n(m)).
		\]
		where $d$ is a total computable function (from, e.g.\ \cite[Lemma II.1.2]{sacks}) which associates $b <_{\ko} b'$ to an index for an $m$-reduction from $\emptyset^{(b)}$ to $\emptyset^{(b')}$ (note that if either $b$ or $b'$ are not in $\ko$, then $d(b,b')$ is an index for an everywhere divergent function). If $n$ witnesses that $\chi(a)$ holds, then, for any given $b,b',m, s$, the computation $\varphi^{f^{[b']}_n}_{d(b,b')}(m)[s]$ is $\varphi^{\emptyset^{[b']}}_{d(b,b')}(m)[s]$, i.e.\ it is a computation with hyperarithmetical oracle, which converges on all inputs. The oracles $\seq{\emptyset^{(b)}}{b <_{\ko} a}$ are uniformly computable in $\emptyset^{(a)}$, so the use functions of these computations are majorized by $g_a \leq^*g$. Using the same trick as in the proof of Theorem \ref{thm:list}, we obtain that $a \in \ko$ is equivalent to the following $\lSigma{0}{2}(x)$ formula:
		\[
		\exists n, k \, \forall b, b', m \,\, ((b<_{\ko} b'<_{\ko}a) \rightarrow \varphi^{f^{[b']}_n}_{d(b,b')}(m)[k+g(m)] \downarrow =f^{[b]}_n(m)).
		\]
		
		For the converse direction, let $f \leq_\T x$ be a function which dominates $\hyp$ and let $h \colon \omega \times \omega \rightarrow \omega$ be an $x$-computable $\lSigma{0}{2}$ approximation of $\ko$, i.e.\ $a \in \ko$ if and only if there is some $m \in \omega$ such that $h(a,n)=1$ for every $n \geq m$. We use $h$ to emulate the construction of the proof of Theorem \ref{thm:list}, building in stages a sequence $\seq{y_{n,m}}{n,m \in \omega}$, as follows.
		
		For any $n$, at step $s+1$, we compute $h(n,s+1)$. If $h(n,s+1)=0$ (i.e.\ if we currently believe that $n \notin \ko$), then we extend all the sequences of the form $y_{n,m}$ with $m < s+1$ by one bit, setting that bit equal to $0$. If $h(n,s+1)=1$, then we modify the sequences $y_{n,m}$ with $m <s+1$ as in the proof of Theorem \ref{thm:list}.
		
		If $n \notin \ko$, then there are infinitely many stages $s$ such that $h(n,s)=0$. It is not hard to see that, in the other stages, nothing is done. This is because the function $c$ of Proposition \ref{prop:complist} has the property that, if $n \notin \ko$, then $c(n)$ is the index of an everywhere diverging function. Therefore, we have that $\seq{y_{n,m}}{m \in \omega}$ is a list of finite sets.
		
		If $n \in \ko$, then there is some $m$ such that $h(n,k)=1$ for every $k \geq m$. The elements of the sequence $\seq{y_{n,k}}{k \geq m}$ coincide with the elements of the corresponding sequence defined in the proof of Proposition \ref{prop:higherjumps}, so that finitely many of these are eventually constant with value $0$, and the other ones are $\emptyset^{(n)}$.
		
		This shows that $\seq{y_{n,m}}{n,m \in \omega}$ consists of some finite sets as well as all the sets of the sequence $\seq{\emptyset^{(a)}}{a \in \ko}$. As in the proof of Proposition \ref{prop:higherjumps}, it follows that $x$ can list $\hyp$.
	\end{proof}
	
	We point out that the right to left direction Theorem \ref{thm:hyplist} is not specific to $\hyp$. Indeed, extending the construction of Proposition \ref{prop:higherjumps} as in the above proof only requires that our function $f$ dominates the uniform moduli $g_i$ of a countable family of functions $x_i$ generating the ideal, and that we have a sufficiently definable way to guess the indices witnessing that $g_i$ is a uniform modulus for $x_i$.
	
	Theorem \ref{thm:hyplist} and a simple application of Gandy's basis theorem gives the separation $\li{\hyp} \subsetneq \wl{\hyp}$.
	
	\begin{theorem}
		$\li{\hyp} \subsetneq \wl{\hyp}$.
	\end{theorem}
	\begin{proof}
		The class of weak lists of $\hyp$ is $\{y \in \baire : \forall x \in \hyp \, \exists n \,\, (y^{[n]}=x)\}$ so, by the Spector--Gandy Theorem, \cite[Theorem III.3.5]{sacks} it is a $\lSigma{1}{1}$ set . By Gandy's basis theorem (\cite[Theorem III.1.4]{sacks}), there is some weak list $y$ of $\hyp$ with $y <_{\hyp} \ko$. On the other hand Theorem \ref{thm:hyplist} implies that any $x \in \li{\hyp}$ is such that $\ko \lhyp x$.
	\end{proof}
	
	We conclude the section showing that $\wl{\hyp} \subsetneq \domm{\hyp}$. This fact was first proved by Noam Greenberg and Dan Turetsky\footnote{Personal communication.} using $\omega$-models of (a sufficiently large fragment of) set theory with nonstandard ordinals. It also follows from results of \cite{GO}. Here we give a direct forcing proof (of the generalization of this fact to ideals closed under $\lhyp$) using some tools from \cite{GO} and diagonalizing against weak listings.  We will see these techniques again in the proof of Theorem \ref{thm:sep}.
	
	\begin{theorem}
		Let $\ideal{I}$ be a countable Turing ideal which is $\lhyp$-downward closed. We have $\domm{\ideal{I}} \nsubseteq \wl{\ideal{I}}$.
	\end{theorem}
	
	\begin{proof}
		The proof uses Hechler tree forcing, with ``small'' bad sets.
		
		We let $\mathbb{H}$ be the set Hechler trees, i.e.\ trees $T \subseteq \omega^{<\omega}$ such that, for every $\rho \succeq \stem{T}$, $\finf n \,\, \concat{\rho}{n} \in T$.
		
		Given $A \subseteq \omega^{<\omega}$ and $\tau \in \omega^{<\omega}$, we say that $A$ is \emph{$\omega$-big above $\tau$} if there is a tree $T$ such that $\stem{T}=\tau$, $[T]= \emptyset$, every non-leaf node on $T$ has infinitely many children, and all leaves of $T$ lie in $A$. If $A$ is not $\omega$-big we say that $A$ is \emph{$\omega$-small}.
		We define the function of $\omega$-closure $\clo$ as $\clo(A)=\{\tau : A \text{ is } \omega\text{-big above } \tau\}$.
		
		In \cite{GO}, it is shown that $\clo(\emptyset)=\emptyset$, and for every $A$ and $B$ (i) $A \subseteq \clo(A)$, (ii) $A \subseteq B \implies \clo(A) \subseteq \clo(B)$, (iii) $\clo(A \cup B)= \clo(A) \cup \clo(B)$ and lastly (iv) $\clo(\clo(A))=\clo(A)$. Moreover, if $A$ is $\lPi{1}{1}(X)$ for some $X$, then $\clo(A)$ is also $\lPi{1}{1}(X)$.
		
		Consider the forcing notion $(\mathbb{P}, \leq)$ given by 
		\begin{align*} 
			\mathbb{P} = \{ (T,B) \in \mathbb{H} \times \mathcal{P}(\omega^{<\omega}) : & \, B \text{ is upwards closed } \land \\ &\exists X \in \ideal{I} \,\, T \text{ is }\lPi{1}{1}(X) \, \land \stem{T} \notin \clo(B)\}
		\end{align*}
		with ordering 
		\[ (T, B) \leq (S, C) \iff T \subseteq S \land B \supseteq C. \]
		
		For a condition $(T,B)$ and a sequence $\rho \in \omega^{<\omega}$, say that $\rho$ is \emph{consistent with $(T, B)$} if and only if, letting $T_{\rho}$ be the maximal subtree of $T$ with stem $\rho$, $(T_\rho, B)$ is a condition. Note that $\rho$ is consistent with $(T,B)$ iff $\rho \in T \setminus \clo(B)$.
		Any sufficiently generic filter on $(\mathbb{P}, \leq)$ defines a real $f_G=\bigcup_{T \in G} \stem{G}$. It is not hard to see that $(T,B) \forces f_G \in [T] \setminus \open{B}$.
		
		Forcing with $(\mathbb{P}, \leq)$ yields a function $f_G$ dominating the functions of $\ideal{I}$: this is because given any $(T, B) \in \mathbb{P}$ and any $h \in \ideal{I}$, there is a condition $(T_h, B) \leq (T, B)$ such that the leftmost path of $T_h$ majorizes $h$.
		We now show that, if $G$ is a sufficiently generic filter on $(\mathbb{P}, \leq)$, then $f_G \notin \wl{\ideal{I}}$.
		
		To that end, fix a condition $(T,B)$, a set $X \in \ideal{I}$ such that $T \leq_{\T} X$ and $B$ is $\lPi{1}{1}(X)$ and a $\{0,1\}$-valued Turing functional $\Phi \in \ideal{I}$. Again, we use $\varphi$ to denote some $\ideal{I}$ computable function on finite strings which gives rise to $\Phi$. We show that, densely below $(T,B)$, we can force that $\Phi(f_G)$ either diverges, or converges to some real which is not a $\{0,1\}$-valued weak list of $\ideal{I}$.\footnote{Here, and in the proof of Theorem \ref{thm:sep}, we represent weak lists of $\ideal{I} \subseteq \baire$ with weak lists of the graphs (in $\cantor$) of elements in $\ideal{I}$, and diagonalize against those.} There are two cases:
		
		\textbf{Case 1:} there is $n < \omega$ and $(S, C) \leq (T, B)$ such that $\stem{S} \notin \clo(D)$, where $D=\{ \rho \in S : \varphi^{\rho}_e(n) \downarrow\}$. In this case $(S, C \cup D)$, is a condition, $(S, C \cup D) \leq (T,B)$ and $(S, C \cup D)\forces \Phi(f_G) \uparrow$, so again, we are done.
		
		\textbf{Case 2:} if we are not in Case 1, then for every $n$ and every $(S,C) \leq (T, B)$ there is $\rho$ consistent with $(S, C)$ such that $\varphi^\rho(n)\downarrow \in \{0,1\}$. In other words, for every $n$ and every $\tau \in T \setminus \clo(B)$, we have $\tau \in \clo(\{\rho \succeq \tau : \varphi^\rho(n) \downarrow \in \{0,1\}\})$. We define, for $i \in \{0,1\}$:
		\[ A_i = \{ (\rho, n) : \rho \text{ consistent with }(T,B) \land \rho \notin \clo (\{\nu \succeq \rho : \varphi^\nu(n) = 1- i\}) \}\]
		Note that the definability of $\clo$ implies that each $A_i$ is $\lSigma{1}{1}(X)$.
		
		We claim that $A_0 \cap A_1 = \emptyset$. Indeed, if $(\rho, n) \in A_0 \cap A_1$, then $\rho$ is consistent with $(T,B)$ and $\rho \notin \clo (\{\nu \succeq \rho : \varphi^\nu(n) = 0\}) \cup
		\clo (\{\nu \succeq \rho : \varphi^\nu(n) = 1\})$. Since $\clo$ commutes with binary union, $\rho \notin \clo(\{\nu \succeq \rho : \varphi^\nu(n) \downarrow \in \{0,1\}\})$ which directly contradicts our hypothesis.
		
		Now since $A_0$ and $A_1$ are $\lSigma{1}{1}(X)$ and disjoint, they have a $\lDelta{1}{1}(X)$ separator $Y$ (\cite[II.3.7]{sacks}). Let $\alpha < \omega^X_1$ be such that $Y \leq_\T X^{(\alpha)}$. We will show that the set of conditions which force that $\Phi(f_G)$ does not list (the characteristic function of) $X^{(\alpha + 1)}$ is dense below $(T, B)$.
		
		So, fix $m$ and $(S, C) \leq (T, B)$ and suppose by contradiction that $(S,C) \forces \Phi(f_G)^{[m]}=X^{(\alpha+1)}$. This means that for every $\rho$ consistent with $(S, C)$ and every $n$, $\varphi^\rho(\langle m,n \rangle) \downarrow \implies \varphi^\rho(\langle m,n \rangle) = X^{(\alpha+1)}(n)$.  We claim that this implies $X^{(\alpha + 1)} \leq_\T Y$, indeed we would have:
		\[n \in X^{(\alpha + 1)} \iff (\stem{S}, \langle m,n \rangle) \in Y.\]
		
		To see this, assume first that $n \in X^{(\alpha+1)}$. Then every string $\rho \succeq \stem{S}$ such that $\varphi^{\rho}(\langle m, n \rangle) \downarrow = 0$ must be inconsistent with $(S,C)$, i.e.\ $\{\rho \succeq \stem{S} : \varphi^{\rho}(\langle m, n \rangle) \downarrow = 0\} \subseteq \clo(C)$. This implies, by definition, that $(\stem{S},\langle m, n \rangle) \in A_1$, so that $(\stem{S},\langle m, n \rangle) \in Y$.
		
		Conversely, if $(\stem{S}, \langle m,n \rangle) \in Y$, then $(\stem{S}, \langle m, n \rangle) \notin A_0$, meaning that $\stem{S}\in \clo(D)$, where $D=\{\rho \succeq \stem{S} : \varphi^{\rho}(\langle m, n \rangle)=0\}$. This means that there must be some $\rho \in D$ which is consistent with $(S,C)$. By assumption $\varphi^{\rho}(\langle m, n \rangle)=X^{(\alpha+1)}(n)$, so $n \notin X^{(\alpha+1)}$.
		
		Putting everything together, if $G$ is a sufficiently generic filter on $(\mathbb{P}, \leq)$, then $f_G \in \domm{\ideal{I}} \setminus \wl{\ideal{I}}$.
	\end{proof}
	
	\section{Separating $\wl{\hyp}$ from $\SNE{\hyp}$}
	
	So far we have seen that $\li{\hyp} \subsetneq \wl{\hyp} \subsetneq \domm{\hyp}$. Results of \cite{GO} mentioned in the introduction give $\SNE{\hyp} \subsetneq \domm{\hyp}$. In this section we show that $\wl{\hyp} \subsetneq \SNE{\hyp}$, concluding the analysis.
	
	We actually prove a more general separation, showing that $\wl{\ideal{I}} \subsetneq \SNE{\ideal{I}}$ holds for every ideal $\ideal{I}$ which is closed under $\lhyp$.
	
	\begin{theorem}\label{thm:sep}
		Let $\ideal{I}$ be a countable Turing ideal which is $\lhyp$-downward closed. We have $\wl{\ideal{I}} \subsetneq \SNE{\ideal{I}}$.
	\end{theorem}
	
	\begin{proof}
		We force over $\ideal{I}$ to obtain a real $x_G$ and a Schnorr name $x \leq_\T x_G$ for a null $\bPi{0}{2}$ set $S \subseteq \cantor$ which covers all the $\ideal{I}$-coded Schnorr null sets, but such that $\ideal{I}(x_G)=\{y \in \baire : y \leq^{\ideal{I}} x_G\}$ does not contain a weak list of $\ideal{I}$, demonstrating $\SNE{\ideal{I}} \setminus \wl{\ideal{I}}$.
		
		In the rest of this proof we adopt the following notational conventions.
		\begin{notation}
			If $\sigma \in \omega^{<\omega}$ is a finite sequence of codes for basic clopens $\seq{C_i}{i < |\sigma|}$ of Cantor space, we abbreviate $\bigcup_{i<|\sigma|}C_i$ as $\hat{\sigma}$. Moreover, for a rational $q \in [0,1]$, we write $R_q$ for the set of names for open subsets of $\cantor$ of measure $q$, i.e.\ $R_q= \{ f \in \baire : \lambda(\open{f}) = q\}$. Note that the sets $R_q$ are uniformly $\lPi{0}{2}$.
		\end{notation}
		
		Let $(\mathbb{P}, \leq)$ be the forcing given by
		\begin{align*}
			\mathbb{P} = \{ (\sigma, R, q) : \,
			& q \in (0, 1/2), \\
			& R \subseteq R_q \text{ is a nonempty $\lSigma{1}{1}(X)$ class for some $X \in \ideal{I}$},\\
			& \sigma \in \omega^{< \omega}, \text{ and}\\
			& \forall f \in R \, (\hat \sigma \subseteq \open{f}) \}.
		\end{align*}
		and ordering
		\[ (\tau, S, r) \leq (\sigma, R, q) \iff 
		\tau \succeq \sigma \, \land \,
		\forall f \in S \, \exists g \in R \,\, (\open{g} \subseteq \open{f}).\]
		Notice that the second condition in particular implies that $r \geq q$.
		
		The following claim establishes some straightforward properties for generics.
		
		\begin{claimn*}
			Let $G$ be a sufficiently generic filter on $\mathbb{P}$ and let $x_G= \bigcup_{p \in G} \sigma^p$.  Then
			\begin{enumerate}
				\item $x_G \in \baire$ and $x_G$ satisfies $\lambda(\open{x_G}) = \frac12$,
				\item $\open{x_G}$ covers all $\ideal{I}$-coded null sets.
			\end{enumerate}
		\end{claimn*}
		
		\begin{proof}[Proof of Claim]
			\begin{enumerate}
				\item Clear.
				\item Let $\seq{(f_i, \lambda_i)}{i < \omega} \in \ideal{I}$ be a Schnorr name for the null set $N=\bigcap_{i \in \omega} \open{f_i}$, $(\sigma, R, q)$ a condition and $i \in \omega$ sufficiently large so that $2^{-i} + q < \frac12$. Then $(\sigma, S, q + 2^{-i})$, where 
				\[S = \{g \in R_{q + 2^{-i}} : \hat{\sigma} \cup \open{f_i} \subseteq \open g \land \exists h \in R\ (\open{h} \subseteq \open{g})\},\]
				is a condition extending $(\sigma, R, q)$.  Moreover, the sets $D_{\tau} = \{ (\rho, T, p) : \hat{\rho} \supseteq [\tau] \}$ are dense above $(\sigma, S, q + 2^{-i})$ for each basic clopen $[\tau]$ in $\open{f_i}$.
				
				This shows that $(\sigma, S, q + 2^{-i}) \Vdash N \subseteq \open{x_G}$, so $\open{x_G}$ covers all $\ideal{I}$-coded Schnorr null sets. 
			\end{enumerate}
		\end{proof}
		
		Now a standard argument, which we sketch for completeness, shows that $x_G$ computes a $\bPi{0}{2}$-name for a null set $N$ which contains all $\ideal{I}$-coded null sets. For every $\sigma \in 2^{<\omega}$ and every $x \in \cantor$, we write $\sigma +x$ to denote the bitwise XOR sum of $\sigma$ and $x$. For any set $A \subseteq \cantor$, we write $\sigma + A$ to denote the set $\{\sigma + x : x \in A \}$.
		
		\begin{claimn*}
			Let $G$ be a sufficiently generic filter on $\mathbb{P}$ and let $A=\bigcap_{\sigma \in 2^{<\omega}} \sigma +\open{x_G}$. Then $A$ is $\ideal{I}(x_G)$-coded $\bPi{0}{2}$, $\lambda(A)=0$, and $N \subseteq A$ for every $\ideal{I}$-coded null set.
		\end{claimn*}
		\begin{proof}[Proof of claim]
			The fact that $A$ is $\ideal{I}(x_G)$-coded and $\bPi{0}{2}$ is clear by definition. The fact that $A$ covers all $\ideal{I}$-coded null sets follows from the fact that $\open{x_G}$ does, and for every $\sigma \in 2^{<\omega}$ and every (measurable) $X \subseteq \cantor$, $\lambda(\sigma + X)=\lambda(X)$. Now, since $A$ is a \emph{tailset} (i.e.\ for every $\sigma,\tau \in 2^{<\omega}$ and every $x \in \cantor$, if $\concat{\sigma}{x} \in A$ and $|\tau|=|\sigma|$, then $\concat{\tau}{x} \in A$) and $\lambda(A) < 1$, by Kolmogorov's 0-1 Law (see, e.g.\ \cite[Theorem 1.2.4]{downeyhirschfeldt}) it follows that $\lambda(A)=0$.
		\end{proof}
		
		We claim that $x_G$ can compute a Schnorr name for $A$. This is due to the fact that all translates of $A$ have measure $\frac12$.
		
		\begin{claimn*}
			There is a Schnorr name $x \leq_{\T} x_G$ for $A$.
		\end{claimn*}
		\begin{proof}[Proof of claim]
			Given an effective enumeration $i$ of $2^{<\omega}$, for every $n$, let $A_n=i(n) + \open{x_G}$. The open sets in the name for $A$ are simply $\seq{B_n}{n \in \omega}$, where $B_n= \bigcap_{m \leq n}A_m$ for every $n$.
			
			We can compute an approximation of $\lambda(B_n)$ from below, uniformly in $n$, by simply going through the names of $\seq{A_m}{m \leq n}$ and computing increasingly tight lower bounds on the measure for their intersection. We can also get approximations from above (uniformly in $n$ again), as follows: to compute an upper bound $b_{n,i}$ for the measure of $B_n$ with $b_{n,i} \geq \lambda(B_n) \geq b_{n_i} -2^{-i}$, it suffices to go through the names for $\seq{A_m}{m \leq n}$ until we have enumerated subsets of measure $\geq \frac12 - 2^{-i}$ of each of the $A_m$'s. 
		\end{proof}		
		The remainder of this proof is dedicated to showing that for every sufficiently generic $G$, $x_G$ does not $\ideal{I}$-compute a weak listing of $\ideal{I}$.
		To that end, we fix a condition $(\sigma, R, q)$ and an $\ideal{I}$-coded $\{0,1\}$-valued Turing functional $\Phi$. Again, we use $\varphi$ to denote the corresponding function on finite strings. We show that, densely below $(\sigma, R, q)$, we can ensure that $\Phi(x_G)$ is not a $\{0,1\}$-valued weak list of $\ideal{I}$. There are two cases:
		
		\noindent\textbf{Case 1.} There is $(\tau, S, r) \leq (\sigma, R, q)$ and $n < \omega$ such that for every $f \in S$ and every $\rho \succeq \tau$, if $\varphi^{\rho}(n) \downarrow = i \in \{0,1\}$, then $\lambda(\hat{\rho} \cup \open{f}) \geq \frac12$. We claim that in this case, we have $(\tau, S, r) \forces \varphi^{x_G}(n) \uparrow$.
		
		Indeed, suppose for contradiction that $G$ is a sufficiently generic filter on $\mathbb{P}$ such that $\varphi^{x_G}(n) \downarrow$, and let $(\rho, T, s) \in G$ with $\varphi^{\rho}(n) \downarrow$. Now take any $g \in T$ and $f \in S$ with $\open{f} \subseteq \open{g}$. We have $\hat{\rho} \subseteq \open{g}$, so that $\hat{\rho} \cup \open{f} \subseteq \open{g}$, and our hypothesis yields $\lambda(\hat{\rho} \cup \open{f}) \geq \frac 12$, hence $\lambda(\open{g}) \geq \frac 12$. This is a contradiction.
		
		\noindent\textbf{Case 2.} Assume there is no $n \in \omega$ for which there is a condition $(\tau, S, r) \leq (\sigma, R, q)$ as above and let $X \in \ideal{I}$ be such that that $R$ is $\lSigma{1}{1}(X)$ and $\varphi$ is $X$-computable. We define, for $i \in \{0,1\}$, 
		\begin{multline*}
			A_i = \bigg\{ (\tau, \varepsilon, n) : 
			\tau \succeq \sigma \land
			2\varepsilon + \lambda(\hat{\tau}) < \frac12 \,  \land \\
			\exists f \in R \, \exists g \, \in R_{\varepsilon + \lambda(\hat{\tau})} \, \bigg(\hat{\tau} \cup 
			\open{f} \subseteq \open{g} \land \forall \rho \succeq \tau \,\, \bigg(\varphi^{\rho}(n) \downarrow = i-1\\
			\rightarrow \lambda(\hat{\rho} \cup \open{g}) \geq \frac12\bigg) \bigg)\bigg\} 
		\end{multline*}
		
		\begin{claimn*}
			$A_0 \cap A_1 = \emptyset$.
		\end{claimn*}
		\begin{proof}[Proof of claim]
			Suppose the contrary, and take $(\tau, \varepsilon, n) \in A_0 \cap A_1$ and $\seq{f_i}{i \in \{0,1\}}$ witnessing $(\tau, \varepsilon, n) \in A_i$.  Let
			\begin{multline*}
				S = \{f \in R_{2\varepsilon + \lambda(\hat{\tau})} : \hat{\tau} \subseteq \open{f}
				\land \exists g \in R \,\, ( \open{g} \subseteq \open{f}) \land \\
				\forall \rho \succeq \tau \,\, (\varphi^{\rho} \downarrow = j \in \{0,1\} \rightarrow \lambda(\hat{\rho} \cup \open{f}) \geq 1/2) \}.
			\end{multline*}
			Then $S$ is $\Sigma_1^1(X)$, and nonempty, since $f_0 \oplus f_1 \in S$.  Thus, $(\tau, S, 2\varepsilon + \lambda(\hat \tau))$ is a condition extending $(\sigma, R, q)$, and $(\tau, S, 2\varepsilon + \lambda(\hat{\tau})) \forces \varphi^{x_G}(n) \uparrow$, contradicting that we are in Case 2.
		\end{proof}
		Since $A_0$ and $A_1$ are $\lSigma{1}{1}(X)$, there is a $\lDelta{1}{1}(X)$ set $B$ with $A_0 \subseteq B$ and $A_1 \cap B = \emptyset$.  Let $\alpha < \omega^X_1$ be such that $B \leq_\T X^{(\alpha)}$. The following two claims allow us to show that $\Phi(x_G)$ does not list the set $X^{(\alpha+1)} \in \ideal{I}$, concluding the proof.
		
		First we note that, without loss of generality, we can assume that in a condition $(\tau, S, r)$, an arbitrarily large portion of $r$ is taken up by the measure of the (open coded by the) ``stem'' $\tau$.
		
		\begin{claimn*}
			For every condition $(\tau, S, r)$ and every $d > 0$, there is a condition $(\tau_d, S_d, r) \leq (\tau, S, r)$ such that $S_d \subseteq S$ and $r-\lambda(\tau_d) \leq d$.
		\end{claimn*}
		\begin{proof}[Proof of claim]
			Fix $d >0$ and take some $f \in S$. We can extend $\tau$ by concatenating it with an initial segment $\nu$ of $X$ which is chosen so that, if $\rho=\concat{\tau}{\nu}$, then $\lambda(\hat{\rho}) \geq r - d$. This is possible because $\hat{\tau} \subseteq \open{f}$ and $\lambda(\open{f})=r$. Then, we can let $S_d=\{g \in S : \hat{\rho} \subseteq \open{g}\}$ and obtain that $S_d$ is nonempty, because $f \in S_d$, so that $(\rho, S_d, r)$ is a condition as above. 
		\end{proof}
		
		\begin{claimn*}
			For each $m < \omega$, we can ensure that the $m$-th column of $\Phi(x_G)$ is not $X^{(\alpha + 1)}$, densely below $(\sigma, R, q)$.
		\end{claimn*}
		
		\begin{proof}[Proof of claim]
			By contradiction suppose this were not the case, so let $m \in \omega$ and $(\tau, S, r) \leq (\sigma, R, q)$ be such that $(\tau, S, r) \Vdash \Phi(x_G)^{[m]}=X^{(\alpha + 1)}$. By the previous claim, we can assume further that $r -\lambda(\hat{\tau}) < 1/2- r$.\footnote{Note that $r-\lambda(\hat{\tau}) > 0$, $r-\lambda(\hat{\tau})+\lambda(\hat{\tau}) = r$, $2(r-\lambda(\hat{\tau}))+\lambda(\hat{\tau})=2r-\lambda(\hat{\tau}) <1/2-r+r=1/2$. }
			
			We show that $X^{(\alpha+1)} \leq_\T B$ by showing that, for every $n$,
			\[n \in X^{(\alpha+1)} \iff \,\, (\tau, r - \lambda(\hat{\tau}), \langle m,n \rangle) \in B,\]
			and this is a contradiction.
			
			If $n \in X^{(\alpha+1)}$, since $(\tau, S, r) \Vdash \Phi(x_G)^{[m]}=X^{(\alpha + 1)}$, we have that for every $\rho \succeq \tau$ and every $f \in S$, if $\open{g} \supseteq \hat{\rho} \cup \open{f}$ and $\varphi^{\rho}(\langle m , n \rangle) \downarrow =0$, then $\lambda(\open{g}) \geq 1/2$ (denote this as $(\star)$). Now let $f \in S$. By definition of extension in $\mathbb{P}$, there is some $g \in R$ with $\open{g} \subseteq \open{f}$ and, by $(\star)$, for any $\rho \succeq \tau$, if $\varphi^{\rho}(\langle m,n \rangle) \downarrow=0$, then $\lambda(\hat{\rho} \cup \open{f}) \geq 1/2$. Since $g \in R$, $f \in S \subseteq R_r$ and $r=\lambda(\hat{\tau})+ r-\lambda(\hat{\tau})$, $g$ and $f$ are witnesses for the fact that $(\tau, r-\lambda(\hat{\tau}), \langle m,n \rangle) \in A_1 \subseteq B$. 
			
			Conversely, assume $(\tau, r-\lambda(\hat{\tau}), \langle m,n \rangle) \in B$, so that, in particular $(\tau, r-\lambda(\hat{\tau}), \langle m,n \rangle) \notin A_0$. By definition of $A_0$ it must be that, for every $g \in R$ and every $f \in R_{r} \supseteq S$ with $\hat{\tau} \cup \open{g} \subseteq \open{f}$, there is some $\rho \succeq \tau$ with $\varphi^{\rho}(\langle m,n \rangle) \downarrow=1$ and $\lambda(\hat{\rho} \cup \open{f}) < 1/2$. This immediately implies that there is some $(\rho, T, t)$ extending $(\tau, S, r)$ with $\varphi^{\rho}(\langle m,n \rangle)=1$, and, since we assumed that $(\tau, S, r) \forces \Phi(x_G)^{[m]}=X^{\alpha+1}$, we get $n \in X^{\alpha+1}$.
		\end{proof}
		Putting everything together, we get that if $x_G$ is sufficiently generic for $\mathbb{P}$, then $\ideal{I}(x_G)$ does not contain a weak list of $\ideal{I}$, concluding the proof. 
	\end{proof}


\begin{thebibliography}{BBTNN15}
		
		\bibitem[BBTNN15]{nies}
		J{\"o}rg Brendle, Andrew Brooke-Taylor, Keng~Meng Ng, and Andr{\'e} Nies.
		\newblock An analogy between cardinal characteristics and highness properties
		of oracles.
		\newblock In {\em Proceedings of the 13th Asian Logic Conference}, pages 1--28.
		World Scientific, 2015.
		
		\bibitem[DH10]{downeyhirschfeldt}
		Rodney~G. Downey and Denis~R. Hirschfeldt.
		\newblock {\em Algorithmic Randomness and Complexity}.
		\newblock Theory and Applications of Computability. Springer, New York, NY,
		2010.
		
		\bibitem[GKT19]{GKT}
		Noam Greenberg, Rutger Kuyper, and Dan Turetsky.
		\newblock Cardinal invariants, non-lowness classes, and {W}eihrauch
		reducibility.
		\newblock {\em Computability}, 8(3-4):305--346, 2019.
		
		\bibitem[GO26]{GO}
		Noam Greenberg and Gian~Marco Osso.
		\newblock Forcing and classes of {H}{Y}{P}-dominating functions.
		\newblock In preparation, 2026.
		
		\bibitem[JJ72]{jockusch}
		Carl~G. Jockusch~Jr.
		\newblock Degrees in which the recursive sets are uniformly recursive.
		\newblock {\em Canadian Journal of Mathematics}, 24(6):1092--1099, 1972.
		
		\bibitem[RJ87]{rogers}
		Hartley Rogers~Jr.
		\newblock {\em Theory of recursive functions and effective computability}.
		\newblock MIT press, 1987.
		
		\bibitem[Rup10]{rupprecht}
		Nicholas Rupprecht.
		\newblock Relativized {S}chnorr tests with universal behavior.
		\newblock {\em Archive for Mathematical Logic}, 49(5):555--570, 2010.
		
		\bibitem[Sac17]{sacks}
		Gerald~E. Sacks.
		\newblock {\em Higher recursion theory}.
		\newblock Perspectives in Logic. Cambridge University Press, 2017.
		
		\bibitem[Sol78]{hypenc}
		Robert~M. Solovay.
		\newblock Hyperarithmetically encodable sets.
		\newblock {\em Transactions of the American Mathematical Society}, 239:99--122,
		1978.
		
	\end{thebibliography}
\end{document}